\documentclass[letterpaper,11pt]{amsart}
\usepackage{amsfonts,amssymb,amsmath,amsthm,latexsym}
\usepackage{stmaryrd}
\usepackage[all]{xy}
\usepackage{fullpage}
\usepackage{hyperref}
\usepackage{enumitem,todonotes}
\usepackage{verbatim}

\usepackage{algorithm}
\usepackage[noend]{algpseudocode}
\usepackage{graphicx}

\usepackage{sistyle}
\SIthousandsep{,}

\hypersetup{
    colorlinks=true,
    linkcolor=blue,
    filecolor=magenta,      
    urlcolor=cyan,      
    citecolor=blue,
}

\DeclareMathOperator{\Tr}{Tr}
\DeclareMathOperator{\Conv}{Conv}

\begin{document}

\makeatletter
\newtheorem*{rep@theorem}{\rep@title}
\newcommand{\newreptheorem}[2]{%
\newenvironment{rep#1}[1]{%
 \def\rep@title{#2 \ref{##1}}%
 \begin{rep@theorem}}%
 {\end{rep@theorem}}}
\makeatother

\newtheorem{thm}{Theorem}
\newreptheorem{thm}{Theorem}
\newtheorem{prop}{Proposition}
\newtheorem{cor}{Corollary}
\newtheorem{lem}{Lemma}
\newreptheorem{lem}{Lemma}

\theoremstyle{definition}
\newtheorem{defn}{Definition}
\newtheorem{conj}{Conjecture}
\newtheorem{prob}{Problem}

\theoremstyle{remark}
\newtheorem{rem}{Remark}
\newtheorem{ex}{Example}
\newtheorem{exer}{Exercise}

\newcommand{\Rho}{\mathrm{P}}
\newcommand{\cS}{\mathcal{S}}
\newcommand{\cM}{\mathcal{M}}
\newcommand{\cN}{\mathcal{N}}
\newcommand{\gk}{\kappa}
\newcommand{\gS}{\Sigma}
\newcommand{\gl}{\lambda}
\newcommand{\gt}{\theta}

\newcommand{\cF}{\mathcal{F}}
\newcommand{\cG}{\mathcal{G}}
\newcommand{\cP}{\mathcal{P}}
\newcommand{\cV}{\mathcal{V}}
\newcommand{\cB}{\mathcal{B}}
\newcommand{\cA}{\mathcal{A}}
\newcommand{\ZZ}{\mathbb{Z}}
\newcommand{\NN}{\mathbb{N}}
\newcommand{\QQ}{\mathbb{Q}}
\newcommand{\bA}{\mathbb{A}}
\newcommand{\bB}{\mathbb{B}}
\newcommand{\bC}{\mathbb{C}}
\newcommand{\bD}{\mathbb{D}}
\newcommand{\bE}{\mathbb{E}}
\newcommand{\bF}{\mathbb{F}}
\newcommand{\bI}{\mathbb{I}}
\newcommand{\bP}{\mathbb{P}}
\newcommand{\bS}{\mathbb{S}}
\newcommand{\fA}{\mathbf{A}}
\newcommand{\fB}{\mathbf{B}}

\newcommand{\RR}{\mathbb{R}}
\newcommand{\CC}{\mathbb{C}}
\newcommand{\FF}{\mathbb{F}}
\newcommand{\HH}{\mathbb{H}}
\newcommand{\PP}{\mathbb{P}}
\newcommand{\EE}{\mathbb{E}}

\newcommand{\cK}{\mathcal{K}}
\newcommand{\cL}{\mathcal{L}}
\newcommand{\cO}{\mathcal{O}}

\newcommand{\fp}{\mathfrak{p}}

\newcommand{\dotcup}{\ensuremath{\mathaccent\cdot\cup}}

\title{A Semidefinite Framework for the Sieve}
\author{Zarathustra Brady}
\email{notzeb@gmail.com}

\thanks{This material is based upon work supported by the NSF Mathematical Sciences Postdoctoral Research Fellowship under Grant No. (DMS-1705177).}

\maketitle

\begin{abstract} We describe a semidefinite programming framework for proving upper bounds on concrete sifting problems, and show that the Large Sieve can be interpreted as a special case of this framework. With a small tweak, the Larger Sieve also falls into this framework.

We compare the semidefinite approach to the linear programming approach (i.e., the general framework of the combinatorial sieve and the Selberg sieve), and show that it has a qualitative advantage in a toy case where the primes are completely independent from each other. No new sieve-theoretic bounds are proved.
\end{abstract}

\section{Introduction}

Let $Z$ be any finite set which we will apply the sieve to. Generally we will imagine that $Z$ is a set of integers - in fact, most often $Z$ will be a set of \emph{consecutive} integers - but our framework applies to any finite set $Z$. We imagine that we are given a collection of partitions $\cP$ of the set $Z$. Generally, we imagine that $\cP$ corresponds to a set of primes $\{p_1, ...\}$, and that the partition corresponding to the prime $p_i$ partitions $Z$ according to the congruence classes of the elements of $Z$ modulo $p_i$. Our basic question is as follows.

\begin{prob} Given a finite set $Z$ and a set of partitions $\cP$ of $Z$, what upper bounds can we place on the size of a set $X \subseteq Z$ such that $X$ misses at least one part of each partition $p \in \cP$?
\end{prob}

More generally, we can imagine that we have a number $\kappa_p$ attached to each partition $p \in \cP$, and we can ask for upper bounds on the sizes of sets $X \subseteq Z$ which avoid at least $\kappa_p$ parts of $p$ for each $p \in \cP$.

As a motivating example of a problem that fits into our framework, we have the problem of finding large admissible tuples.

\begin{defn} A set $H \subseteq \ZZ$ is \emph{admissible} if for all primes $p$, $H$ avoids at least one congruence class modulo $p$.
\end{defn}

Taking $Z = \{0,...,N-1\}$ and $\cP$ to be the set of partitions of $Z$ into congruence classes modulo primes $p \le N$, we see that finding the largest admissible tuple $H \subseteq \{0, ..., N-1\}$ is a special case of our general problem.

We approach this question from a somewhat unusual point of view. We imagine that we are handed a relatively small set $Z$ and collection of partitions $\cP$ (i.e., $|Z|, |\cP| < 10^{10}$), and wish to use a computer to automatically prove a numerical upper bound on the size of $X$. Finding the exact optimal upper bound by brute force is no good - we want an \emph{efficient} algorithm for finding upper bounds. Failing that, we would be satisfied with proofs of upper bounds which can be efficiently and mechanically verified, but which may be very difficult to find. Ideally, proofs should be put into a standard form so that examples of proofs that work well can be compared to each other and, hopefully, generalized. This second requirement leads us to consider specialized proof frameworks which are less general than, say, all of axiomatic set theory.

In Section \ref{s-linear} we will review the standard framework of systems of linear inequalities, together with a choice of \emph{sieve weights} which are used to produce a proof of an upper bound on the size of a set $X \subseteq Z$ avoiding at least one congruence class from each partition in $\cP$. This framework has several drawbacks - most importantly for us, it is not clear in general if one can efficiently verify that a collection of sieve weights truly leads to a valid inequality. In practice, this difficulty is avoided by making special choices of sieve weights, for which specialized arguments can be used to prove that they lead to valid inequalities.

After the review of the linear approach we will introduce a new framework of systems of quadratic inequalities in Section \ref{s-framework}, which we will represent using \emph{positive semidefinite matrices}. A proof in this framework is just a collection of matrices satisfying certain positive semidefiniteness conditions on their submatrices, such that their sum is a small multiple of the identity matrix. The entries of these matrices are analogous to the sieve weights of the linear framework.

The advantage of the semidefinite approach is that it is easy to mechanically check whether a matrix is positive semidefinite - for instance, one can simply compute the Cholesky decomposition of the matrix (in terms of the original quadratic inequalities, we can think of this as trying to prove a quadratic inequality by repeatedly completing the square). In fact, standard semidefinite programming algorithms allow one to compute the best possible upper bound using this framework in polynomial time (in numerical experiments, however, the standard semidefinite solvers have not been able to handle sets $Z$ with $|Z| \gg 10^3$ in a reasonable amount of time).

In Sections \ref{s-large} and \ref{s-larger}, we will show that the Large Sieve and the Larger Sieve both fit into our general semidefinite framework, although the framework needs to be expanded slightly to handle the Larger Sieve. The special form of the semidefinite matrices used in these two cases leads to the consideration of a simpler (but slightly less powerful) semidefinite framework which is described in Section \ref{s-simpler}.

In Section \ref{s-first-level}, we show that there is a sense in which we can combine inequalities which each consider just a single partition on its own to produce strong upper bounds on the size of the sifted set $X$. This can be contrasted with the case of the linear framework, where considering one partition at a time only leads to a nontrivial upper bound if $\sum_{p\in \cP} \frac{1}{|p|}$ is less than $1$.


\subsection{Notation}

The author has attempted to find notation which reflects the general nature of the problem, but which matches with the usual multiplicative notation used in sieve theoretic arguments as closely as possible. We therefore always denote partitions in $\cP$ with lowercase letters $p$ or $q$, and we describe subsets $d$ of $\cP$ with a multiplicative notation. Thus $1$ corresponds to the empty subset of $\cP$, and we think of $d\cup k, d\cap k$ as corresponding to the lcm and gcd of $d$ and $k$. If $d,k \subseteq \cP$ are disjoint, then we write $dk$ for their disjoint union. We also abuse notation by writing $p$ for the singleton subset $\{p\} \subseteq \cP$, so if $p \in \cP, d \subseteq \cP$ then the expression $pd$ corresponds to the subset $d \cup \{p\}$ of $\cP$.

For divisibility, the expression $p \in d$ corresponds to $p$ being a prime dividing $d$. The expression $k \subseteq d$ corresponds to $k$ being a divisor of $d$. The notation $d\setminus k$ for the set-theoretic difference is a curious anomaly: if $k \subseteq d$, then it corresponds to $\frac{d}{k}$, and in general it corresponds to $\frac{d}{\gcd(d,k)}$.

When working in the general setup, we use $|p|$ to denote the number of parts in the partition $p$. We extend $\kappa$ to a multiplicative function on $d \subseteq \cP$ by $\kappa(d) = \prod_{p\in d} \kappa_p$. We also define multiplicative functions $\phi(d), \mu(d)$ for $d \subseteq \cP$ by $\phi(p) = |p|-1$ and $\mu(p) = -1$ in analogy with the usual Euler $\phi$-function and the M\"obius function, although we will later need to introduce variants of the $\phi$ function which take the numbers $\kappa_p$ as parameters.

A multiplicative notation for matrices will be introduced in Section \ref{s-large}, and this notation will prove useful in all later sections.

\section{The Linear Framework (and its difficulties)}\label{s-linear}

We will suppose throughtout this section that the partitions in $\cP$ correspond to primes $p$, in order to make use of the existing multiplicative notation of divisibility and squarefree numbers, rather than introducing new notation for sets of partitions. Suppose that for each prime $p$, $X \subseteq Z$ avoids a set $A_p \subseteq Z$ corresponding to some congruence class modulo $p$. If we are working in the more general setting with a number $\kappa_p$ of congruence classes to be avoided modulo each $p$, we assume that $A_p \subseteq Z$ corresponds to some collection of $\kappa_p$ congruence classes modulo $p$.

We extend the notation $A_p$ to squarefree numbers $d$ multiplicatively:
\[
A_d = \bigcap_{p\mid d} A_p,
\]
where we interpret $A_1$ as $Z$. The linear framework assumes as given a system of linear inequalities
\[
-R_d \le |A_d| - \frac{\kappa(d)}{d}|Z| \le R_d,
\]
where $R_d$ are \emph{remainder terms} which are supplied to us by an outside source, and $\kappa(d)$ is defined by
\[
\kappa(d) = \prod_{p \mid d} \kappa_p.
\]
The reader should keep in mind that the main case of interest has $\kappa(d) = 1$ for all $d$.

Adding together the given inequalities with \emph{sieve weights} $\lambda_d$, we get the inequality
\[
\sum_d \lambda_d|A_d| = \Big(\sum_d \lambda_d \frac{\kappa(d)}{d}\Big)|Z| + \cO^*\Big(\sum_d |\lambda_d| R_d\Big),
\]
where the notation $f = \cO^*(g)$ is a shorthand for the inequality $-g \le f \le g$. In order for this to give an upper bound on $|X|$, we need the following implication to be valid:
\[
X \subseteq A_1, \forall p\ X \cap A_p = \emptyset \;\; \implies \;\; |X| \le \sum_d \lambda_d |A_d|.
\]
This implication is equivalent to $\lambda_1 \ge 1$ together with the system of inequalities
\[
\forall k\ \ \sum_{d \mid k} \lambda_d \ge 0.\tag{$*$}\label{sieve-weights}
\]

So the general linear framework proceeds as follows. First we estimate the remainder terms $R_d$. Then we choose a collection of sieve weights $\lambda_d$ which satisfy \eqref{sieve-weights} and have $\lambda_1 = 1$. Finally, we deduce the upper bound
\[
|X| \le \Big(\sum_d \lambda_d \frac{\kappa(d)}{d}\Big)|Z| + \sum_d |\lambda_d| R_d.
\]

The biggest difficulty with the linear framework is that we need to verify the system of inequalities \eqref{sieve-weights}. Even if the $\lambda_d$ are supported on values of $d$ with $d \ll |Z|$, we still have to check \eqref{sieve-weights} for \emph{all} $k$ dividing $\prod_{p \le |Z|} p$.

\begin{prop} It is co-NP-hard to check whether a given collection of sieve weights $\lambda_d$ satisfies \eqref{sieve-weights}, even if $\lambda_d$ is supported on $d$ with at most two prime factors.
\end{prop}
\begin{proof} In fact, this problem is equivalent to determining whether a quadratic polynomial ever takes a negative value, when its inputs are restricted to values from $\{0,1\}$: if we write $k = \prod_p p^{x_p}$, then $\sum_{d \mid k} \lambda_d = \sum_{p < q} \lambda_{pq} x_px_q + \sum_p \lambda_p x_p + \lambda_1$ (note that $x_p^2 = x_p$ for $x_p \in \{0,1\}$, so there is no loss of generality in considering quadratic polynomials of this form). This problem is known as the Binary Quadratic Programming problem, and is well-known to be hard. For instance, we can reduce from $k$-clique as follows: let $G$ be a graph with vertices labeled by primes $p$, put $\lambda_1 = 1$, $\lambda_p = -\frac{1}{k-1}$, and $\lambda_{pq} = |G|$ for any $p,q$ which are not connected by an edge of the graph $G$.
\end{proof}

For computing asymptotics, a good strategy is to bucket the large primes $p$ into finitely many buckets $B_i$, such that each bucket has $\sum_{p \in B_i} \frac{1}{p}$ small (the small primes can be handled by a Selberg sieve, and do not contribute much to the asymptotics). Then we decide on values for $\lambda_d$ which only depend on the number of primes from each bucket which divide $d$, so we can write
\[
\lambda_d = \lambda(e_1, ..., e_k), \;\;\; e_i = \#\{p \in B_i \text{ s.t. } p \mid d\}.
\]
The system of inequalities \eqref{sieve-weights} then becomes
\[
\forall n_1, ..., n_k \in \NN \;\;\; \sum_{e_1, ..., e_k} \lambda(e_1, ..., e_k) \binom{n_1}{e_1}\cdots \binom{n_k}{e_k} \ge 0,
\]
that is, we must determine whether a given polynomial in $k$ variables is nonnegative when its variables take values from the natural numbers.

\begin{prop} If $k$ is sufficiently large, then there is no algorithm which determines whether a given polynomial $f \in \ZZ[x_1, ..., x_k]$ is nonnegative on the naturals, even if the degree of $f$ is bounded by $8$.
\end{prop}
\begin{proof} By Matiyasevich's resolution of Hilbert's Tenth Problem \cite{matiyasevich}, if $k$ is sufficiently large then there is no algorithm which determines whether a given polynomial $g \in \ZZ[x_1, ..., x_k]$ ever takes the value $0$ for natural inputs $x_1, ..., x_k$, even if the degree of $g$ is bounded by $4$. Now take $f(x_1, ..., x_k) = g(x_1, ..., x_k)^2 - 1$, and note that $f$ is nonnegative on the naturals if and only if $g$ never takes the value $0$ for natural inputs.
\end{proof}

Despite the above result, it turns out that by slightly increasing the high-order terms of our polynomial $\sum_{e_1, ..., e_k} \lambda(e_1, ..., e_k) \binom{n_1}{e_1}\cdots \binom{n_k}{e_k}$ we can reduce proving the inequality to a finite search together with a proof that the polynomial goes to infinity as the variables go to infinity, at the cost of slightly decreasing the quality of our asymptotic bounds. This enabled Selberg \cite{selberg} to prove that optimal asymptotics for the linear approach to the sieve are at least \emph{computable} in theory.

In practice, rather than compute the optimal choice of sieve weights, most work on the linear approach to the sieve focuses on sets of sieve weights which are guaranteed to work by some simple principle. The most prominent example is the principle behind the Selberg sieve: at a high-level, rather than trying to optimize over the set of polynomials which are nonnegative on the natural numbers, we try to optimize over the set of polynomials which can be written as a sum of squares (and are thus nonnegative on the reals as well). The alternative approach is to use recursive principles such as Buchstab iteration, which are guaranteed to produce new valid sets of sieve weights from old valid sets of sieve weights.

There are a few additional drawbacks to the linear framework from the point of view of the general problem considered in this paper. The first is that the requirement of bounds of the form
\[
|A_d| = \frac{\kappa(d)}{d}|Z| + \cO^*(R_d)
\]
is very inflexible. In cases where the parts of some partition $p \in \cP$ have very different sizes from each other, the bounds we get fail to degrade gracefully. A more detailed linear relaxation (with different sieve weights for each element of $Z$) could overcome this difficulty.

The next drawback we will mention is related to the concept of the \emph{hierarchy} of successively stronger linear relaxations to a problem. The \emph{basic} LP would only allow us to use sieve weights $\lambda_d$ with $d$ equal to a prime. At the second level of the hierarchy, we would consider sieves with $\lambda_d$ supported on $d$ a product of two primes, and so on. A sieve which is low on this hierarchy is considered logically simpler than a sieve which is high on this hierarchy. Unfortunately, in the case of ordinary sieving, any constant level of this hierarchy of linear relaxations is useless.

\begin{prop} If $\sum_p \frac{\kappa_p}{p}$ diverges, then for any set of sieve weights $\lambda_d$ satisfying \eqref{sieve-weights} (and with $\lambda_1 = 1$) which are supported on the set of $d$ with at most $2m$ prime factors, we have
\[
\frac{\sum_d \lambda_d\frac{\kappa(d)}{d}}{\prod_p \big(1 - \tfrac{\kappa_p}{p}\big)} \rightarrow \infty.
\]
\end{prop}
\begin{proof} Letting $\theta(k) = \sum_{d \mid k} \lambda_d$, we have
\[
\prod_p \Big(1 - \frac{\kappa_p}{p}\Big) \sum_k \theta(k)\frac{\kappa(k)}{k} = \sum_d \lambda_d\frac{\kappa(d)}{d}
\]
by M\"obius inversion. So we just have to show that $\sum_k \theta(k)\frac{\kappa(k)}{k} \rightarrow \infty$. By \eqref{sieve-weights}, each $\theta(k)$ is $\ge 0$, so we just have to show that enough of them are sufficiently large to finish.

For any $k$ having $2m+1$ prime factors, we have
\[
0 = \lambda_k = \sum_{d\mid k} \mu(d)\theta(d)
\]
by M\"obius inversion again, so
\[
\sum_{d\mid k,\ \mu(d)=-1} \theta(d) = \sum_{d\mid k,\ \mu(d)=1} \theta(d) \ge \theta(1) = 1.
\]
Thus, for every $k$ with $2m+1$ prime factors, $k$ has some nontrivial divisor $d$ with $\theta(d) \ge 2^{-2m}$. Call $d$ ``good'' if $\theta(d) \ge 2^{-2m}$, otherwise call it ``bad''.

Suppose for contradiction that $\sum_k \theta(k)\frac{\kappa(k)}{k}$ remains bounded. Then in particular $\sum_{d\text{ good}} \frac{\kappa(d)}{d}$ remains bounded, so there is some bounded prime $p_1$ which is bad. Call $d$ ``$p_1$-good'' if $p_1 \nmid d$ and one of $d, p_1d$ is good. Then $\sum_{d\ p_1\text{-good}} \frac{\kappa(d)}{d}$ remains bounded, so there is a bounded $p_1$-bad prime $p_2$. Continuing like this, we see that there is some bounded $k$ with $2m+1$ prime factors such that every nontrivial divisor of $k$ is bad, a contradiction.
\end{proof}

As a consequence, any set of sieve weights which work well for the ordinary sifting scenario is necessarily somewhat intricate, with the support of the sieve weights depending on the sizes of the prime factors as well as their number.

Finally, there is the famous \emph{parity problem} identified by Selberg \cite{selberg} in the case where all the $\kappa_p$ are $1$. The analysis goes roughly as follows. First, by analyzing the identity
\[
\prod_p \Big(1 - \frac{1}{p}\Big) \sum_k \frac{\theta(k)}{k} = \sum_d \frac{\lambda_d}{d}
\]
used in the proof of the previous proposition, we see that if the ratio $\frac{\sum_d \lambda_d/d}{\prod_p (1 - 1/p)}$ is reasonably small, then the average values of the $\theta(k)$s (weighted by $\frac{1}{k}$) must also be small. By the M\"obius inversion formula, the $|\lambda_d|$s must be small on average as well.

So long as the support of the $\lambda_d$s is reasonable, this means that the contribution of the remainder terms in our bound
\[
|X| \le \Big(\sum_d \frac{\lambda_d}{d}\Big)|Z| + \sum_d |\lambda_d| R_d.
\]
is small compared to the main term even if the remainders get somewhat large, especially if we consider using the same collection of sieve weights on a slightly larger set $Z$.

The fact that the bounds do not depend much on the remainder terms means that the linear framework can't make very effective use of tight bounds on the remainder terms $R_d$ - bounds which are $O(1)$ do not lead to substantially better asymptotics than bounds which are $O((|Z|/d)^{1-\epsilon})$. So we see that it is possible to perturb the set $Z$ by removing all numbers from $Z$ which have an even number of prime factors and double counting the rest, without changing the fact that the remainder terms $R_d$ are $O((|Z|/d)^{1-\epsilon})$. Since the upper bound applies to the perturbed case as well, the best possible upper bounds in this framework are necessarily a factor of $2$ larger than the true bound if the support of the sieve weights is reasonable, even if the remainders are completely nonexistent.

\section{The semidefinite framework}\label{s-framework}

The main idea is to try to prove an inequality of the form
\[
\Big(\sum_{i \in X} x_i\Big)^2 \le \lambda\sum_{i \in X} x_i^2,
\]
where $X$, the support of variables $x_i$, avoids at least one part of each partition $P_i$. Taking $x_i = 1$ for $i \in X$, this will prove the inequality $|X| \le \lambda$.

In order to prove an inequality of this form, we try to write the expression $\lambda\sum_{i} x_i^2 - (\sum_{i} x_i)^2$ as a sum of ``obvious'' inequalities $Q^d(\overline{x}) \ge 0$, where the quadratic form $Q^d$ has the property that expressions of the form
\[
Q^d(x_1,x_2, 0, x_4, ...)
\]
are positive semidefinite, where $0$s have been inserted into indices which fall into the parts of the partitions corresponding to $d$ which we are avoiding (recall that we are using a multiplicative notation, so $d$ corresponds to a collection of partitions), for every way of choosing one part of each partition from $d$ to avoid. When dealing with the Larger Sieve, it will also be helpful to include an additional inequality of the form $Q^0(\overline{x}) \ge 0$, where the only assumption on $Q^0$ is that its coefficients are nonnegative.

When manipulating quadratic forms on a computer, it is natural to organize their coefficients into symmetric matrices. To each quadratic form $Q^d$ we associate a symmetric matrix $B^d$ - we think of the system of matrices $B^d$ as an analogue of a system of sieve weights $\lambda_d$ from the linear framework. Let $J$ be the all-ones matrix. Our proof of an upper bound on $X$ will now take the form of a chain of inequalities:
\[
x^TJx \le \sum_d x^TB^dx + x^TJx = \lambda x^Tx,
\]
where each $x^TB^dx \ge 0$ can be easily checked in isolation.

Note that for any $B,x$ we have $x^TBx = \Tr(Bxx^T)$, so it is natural to write $A = xx^T$ and rewrite our proof in the form
\[
\Tr(JA) \le \sum_d \Tr(B^dA) + \Tr(JA) = \lambda\Tr(A).
\]
In this more general form, the positive semidefinite matrix $A$ has a new interpretation. We can imagine having a probability distribution $\mu$ on the collection of all subsets $X \subseteq Z$, and let $A$ be the covariance matrix:
\[
A_{ij} = \bP_\mu[i \in X\ \wedge\ j \in X].
\]
To capture the fact that the unknown set $X$ must avoid at least one part of each partition, we require that for many choices of $d \subseteq \cP$, the matrix $A$ is in the convex hull of matrices of the form $xx^T$, where the vector $x$ is supported on a set which avoids at least one part of each partition $p \in d$.

An important special case to think about is the case where we just have one partition, and each part of the partition has size $1$. For the sake of concreteness, imagine that our universe $Z$ is $\ZZ/p$ for some $p$. We define two closed convex cones:
\begin{align*}
\mathcal{A}_p &= \Conv\{xx^T\text{ s.t. }x \in \RR^p,\ x_i = 0\text{ for some }i\in \ZZ/p\},\\
\mathcal{B}_p &= \{B \in \RR^{p\times p}\text{ s.t. }B|_{Z\setminus \{i\}}\succeq 0\text{ for all }i\in \ZZ/p\},
\end{align*}
where $B|_{Z\setminus \{i\}}$ is a submatrix of $B$ obtained by deleting the $i$th row and column of $B$, and $M \succeq 0$ means that the matrix $M$ is positive semidefinite.

\begin{prop} The cones $\cA_p, \cB_p$ defined above are dual:
\begin{align*}
B \in \cB_p &\iff \forall A \in \cA_p,\ \Tr(AB) \ge 0,\\
A \in \cA_p &\iff \forall B \in \cB_p,\ \Tr(AB) \ge 0.
\end{align*}
Furthermore, for any matrices $A,B$ we can efficiently check whether $B \in \cB_p$ or $A \in \cA_p$. For $p \ge 3$, both $\cA_p$ and $\cB_p$ contain a neighborhood of the identity matrix.
\end{prop}
\begin{proof} A matrix $B$ is in $\cB_p$ iff for every $i$ and every $x \in \RR^p$ with $x_i = 0$, we have $x^TBx \ge 0$, and from $x^TBx = \Tr(xx^TB)$ we get the first duality. We can check whether $B \in \cB_p$ by computing Cholesky decompositions of the matrices $B|_{Z\setminus \{i\}}$ for each $i \in \ZZ/p$, this takes time roughly $O(p\cdot p^3) = O(p^4)$.

The statements about $\cA_p$ can be deduced from the statements about $\cB_p$ by conic duality, so long as we can show that $\cA_p$ is closed. Concretely, a matrix $A$ is in $\cA_p$ iff there are matrices $X_1, ..., X_p$ such that
\begin{itemize}
\item $X_i \succeq 0$ for all $i$,

\item the $i$th row and column of $X_i$ is all $0$s, and

\item $A = \sum_i X_i$.
\end{itemize}
Since each $X_i \preceq A$, if $A$ is from a bounded set then the possible $X_i$s which we might consider also come from from a bounded set, and this together with the fact that $X_i \succeq 0$ and $\sum_i X_i = A$ are closed conditions imply that $\cA_p$ is closed.

We can check whether a matrix $A \in \cA_p$ by solving the semidefinite program corresponding to the three bullet points above. Alternatively, we can check $A \in \cA_p$ by minimizing $\Tr(AB)$ over $B \in \cB_p$ (perhaps with some extra constraint such as $\|B\| \le 1$).

That $\cB_p$ contains a neighborhood of the identity matrix is clear from the definition. For $\cA_p$, we use the fact that the set of $X_i$ as in the bullet points forms a neighborhood of $\frac{1}{p-1}I|_{Z\setminus\{i\}}$ in $\RR^{(p-1)\times(p-1)}$, and the sum of these neighborhoods forms a neighborhood of $I$ for $p \ge 3$.
\end{proof}

Returning to the general case, we can define cones $\cA_d, \cB_d$ for any set $d$ of partitions from $\cP$.

\begin{defn} For $d = \{p_1,...,p_k\} \subseteq \cP$, where $\cP$ is a set of partitions of $Z$, we define the cones $\cA_d, \cB_d$ by
\begin{align*}
\mathcal{A}_d &= \Conv\{xx^T\text{ s.t. }x \in \RR^Z,\ \forall p_i \in d\ \exists c_i \in p_i\text{ with }x_j = 0\text{ for }j\in c_i\},\\
\mathcal{B}_d &= \{B \in \RR^{Z\times Z}\text{ s.t. }\forall c_1 \in p_1, ..., c_k \in p_k\ B|_{Z\setminus (c_1\cup \cdots \cup c_k)}\succeq 0\},
\end{align*}
where $B|_{Z\setminus (c_1\cup \cdots \cup c_k)}$ is the matrix formed by deleting all rows and columns in $c_1\cup \cdots \cup c_k$ from $B$.
\end{defn}

A similar argument to the previous proposition shows that $\cA_d, \cB_d$ are dual for each $d$. The general result we will use is the following application of semidefinite duality.

\begin{thm}\label{duality} If $D = \{d_1, ...\}$ is any collection of subsets of $\cP$, $\cP$ a set of partitions of $Z$, then
\begin{equation}
\max_{A \in \bigcap_{d \in D} \cA_d} \frac{\Tr(AJ)}{\Tr(A)} = \min_{B^{d_i} \in \cB_{d_i}} \big\|\sum_{d\in D} B^d + J\big\|_{op}.\tag{$*$}\label{target}
\end{equation}
The minimum is attained with a collection of matrices $B^{d_i}$ such that $\sum_{d \in D} B^d + J$ is a multiple of the identity matrix.
\end{thm}
\begin{proof} First we show that the left hand side is at most the right hand side. If $A \in \bigcap_{d\in D} \cA_d$, then for each $B^{d_i} \in \cB_{d_i}$ we have $\Tr(AB^{d_i}) \ge 0$, so
\[
\Tr(AJ) \le \sum_{d\in D} \Tr(AB^d) + \Tr(AJ) = \Tr\Big(\Big(\sum_{d\in D} B^d + J\Big)A\Big) \le \big\|\sum_{d\in D} B^d + J\big\|_{op}\Tr(A),
\]
where the last inequality follows from the fact that $|\Tr(MA)| \le \|M\|_{op}\Tr(A)$ for any square matrix $M$ and any positive semidefinite matrix $A$ (to see this, just represent $A$ as a sum of outer products $xx^T$).

For the other direction, we will use the fact that the cone $\bigcap_{d\in D} \cA_d$ is dual to the cone $\sum_{d \in D} \cB_d$, which will follow from the fact that each pair $\cA_{d_i}, \cB_{d_i}$ are dual as long as we can show that $\sum_{d \in D} \cB_d$ is closed. In the case where no partition $p \in d_i \in D$ has just two parts, this follows from the facts that each $B^{d_i} \in \cB_{d_i}$ has nonnegative entries along the diagonal and that the off-diagonal entries are bounded in terms of the diagonal entries. If some $p \in d_i \in D$ has just two parts, then some of the entries of matrices $B^{d_i} \in \cB_{d_i}$ can take arbitrary values independently of the other entries, so these entries may be ignored, and all other entries of $B^{d_i}$ are again bounded in terms of the diagonal entries.

Note that since $\Tr(A) \ge 0$ for $A \in \bigcap_{d\in D} \cA_d$, the maximum $\max_{A \in \bigcap_{d \in D} \cA_d} \frac{\Tr(AJ)}{\Tr(A)}$ is the same as
\[
\max_{\substack{A \in \bigcap_{d \in D} \cA_d,\\ \Tr(A) = 1}} \Tr(AJ),
\]
and since the set $\{A \in \bigcap_{d \in D} \cA_d,\ \Tr(A) = 1\}$ is compact and nonempty, the maximum is actually attained. By the duality between $\bigcap_{d\in D} \cA_d$ and $\sum_{d \in D} \cB_d$, the maximum above is equal to the minimum $\lambda \in R$ such that $\lambda I - J \in \sum_{d\in D} \cB_d$ (and this minimum is attained because $\sum_{d\in D} \cB_d$ is closed). For this minimum value of $\lambda$, there exist $B^{d_i}$ with
\[
\sum_{d\in D} B^d + J = \lambda I.\qedhere
\]
\end{proof}

\begin{cor} If some $d \in D$ is equal to all of $\cP$, then \eqref{target} is equal to the size of the largest set $X \subseteq Z$ which avoids at least one part of each partition $p \in \cP$.
\end{cor}
\begin{proof} We just have to check that for any $A \in \cA_{\cP}$, we have $\Tr(AJ) \le \lambda \Tr(A)$, where $\lambda$ is the size of the largest such set $X$. Since each $A \in \cA_{\cP}$ can be written as a sum of rank one matrices $xx^T$ with $x$ supported on a set $X$ which avoids at least one part of each partition $p \in \cP$, we just need to show that if $x$ is supported on $X$ then $x^TJx \le |X|x^Tx$, that is, that $(\sum_{i\in X} x_i)^2 \le |X|\sum_{i \in X}x_i^2$. This follows from Cauchy-Schwarz.
\end{proof}

Of course, it is impractical to compute the sets $\cA_\cP, \cB_\cP$ if $\cP$ is large. So our strategy is to pick some collection $D = \{d_1, ...\}$ of \emph{small} subsets of $\cP$, and hope that the intersection $\bigcap_{d\in D} \cA_d$ is a good approximation to $\cA_{\cP}$. Then \eqref{target} can be efficiently computed to any desired accuracy (in theory), and sets of witnessing matrices $B^{d_i}$ can be produced to give a proof of an upper bound.

\begin{prop} If $D$ is any collection of subsets of $\cP$, then \eqref{target} can be computed to any desired accuracy $\epsilon$ in time polynomial in $\log(1/\epsilon)$, $|Z|$, and $\sum_{d \in D} \prod_{p \in d} |p|$.
\end{prop}
\begin{proof} This follows from a general result of Khachiyan \cite{ellipsoid}, which reduces the task of minimizing a convex function over a closed convex set to the task of testing whether a point is in the set. To test whether a matrix $B^d \in \cB_d$, we just need to compute a Cholesky decomposition of $\prod_{p \in d} |p|$ submatrices of $B$.
\end{proof}

All of the above results and definitions can be extended to the case where we sieve out $\kappa_p$ parts of the partition $p$ instead of just sieving out one part, but there is an annoying additional complication if $\kappa_p$ is not bounded. To understand this complication, we again reduce to a simple case where we have just one partition $p$, where each part of $p$ has just one element. We define closed convex cones
\begin{align*}
\mathcal{A}_{p,\kappa} &= \Conv\{xx^T\text{ s.t. }x \in \RR^p,\ x_{i_1} = \cdots = x_{i_\kappa} = 0\text{ for some distinct }i_1, ..., i_\kappa\in p\},\\
\mathcal{B}_{p,\kappa} &= \{B \in \RR^{p\times p}\text{ s.t. }B|_{Z\setminus \{i_1,...,i_\kappa\}}\succeq 0\text{ for all distinct }i_1, ..., i_\kappa\in p\},
\end{align*}
and note that these are again dual to each other.

\begin{prop} If $p$ and $\kappa$ are allowed to have the same order of magnitude, then it is NP-hard to test whether a given matrix $B$ is contained in $\cB_{p,\kappa}$.
\end{prop}
\begin{proof} Checking whether $B \in \cB_{p,\kappa}$ is the same as determining whether there exists $x \in \RR^p$ having at most $p-\kappa$ nonzero entries and satisfying $x^TBx < 0$. If we could solve this efficiently for all $\kappa$, then we could find the sparsest vector $x$ which satisfies $x^TBx < 0$. We will show this is NP-hard by a reduction from the problem of finding the sparsest vector that approximately solves a linear system, aka the ``best subset selection problem'' (see \cite{sparse-approximate-hardness}).

The reduction from the best subset selection problem to our problem goes as follows: if it is hard to find a sparse $x$ satisfying $\|Ax - b\| < 1$, then it is equally hard to find a sparse vector $(x_0, x_1, ..., x_n)$ satisfying $\|A(x_1, ..., x_n)^T - bx_0\|^2 - x_0^2 < 0$, and the left hand side is a quadratic form in $x_0, ..., x_n$.
\end{proof}

In practice, when $\kappa_p$ is allowed to be large, we get around this difficulty by restricting attention to special matrices $B^p$ which have a symmetry property that allows us to efficiently verify that $B^p \in \cB_{p,\kappa_p}$.

\section{The Large Sieve}\label{s-large}

To write down the Large Sieve compactly, it is efficient to use a multiplicative notation for matrices. We think of our interval $Z$ as a subset of $\prod_p \ZZ/|p|$. We use $I_p, J_p$ for the $p\times p$ identity matrix and matrix of all ones, respectively, and if $M_p \in \RR^{p\times p}, N_q \in \RR^{q\times q}$, then we write $M_p \otimes N_q$ for their tensor product, thought of as a matrix with rows and columns indexed by elements of $\ZZ/|p| \times \ZZ/|q|$. In particular, we have
\[
I_p \otimes I_q = I_{pq}, \;\;\; J_p \otimes J_q = J_{pq}.
\]
The matrices we write down with this multiplicative notation will, strictly speaking, have rows and columns indexed by elements of $\prod_p \ZZ/p$ which lie outside the interval $Z$. We ignore these extra rows and columns - that is, we consider the natural projection map $\RR^{\prod_p p} \rightarrow \RR^Z$ and the inclusion $\RR^Z \rightarrow \RR^{\prod_p p}$ corresponding to $Z \rightarrow \prod_p \ZZ/p$, and pre- and post-compose with these to obtain matrices in $\RR^{Z\times Z}$.

\begin{prop} If $A,B,C,D$ are symmetric matrices with $A \succeq B \succeq 0$ and $C \succeq D \succeq 0$, then $A \otimes C \succeq B \otimes D \succeq 0$.
\end{prop}
\begin{proof} To see that the tensor product of positive semidefinite matrices is positive semidefinite, just note that every positive semidefinite matrix is a positive combination of rank one matrices of the form $xx^T$, and that tensor products of such matrices are again of that form. To finish, note that
\[
A\otimes C - B\otimes D = A\otimes(C-D) + (A-B)\otimes D \succeq 0.\qedhere
\]
\end{proof}

\begin{lem}[Montgomery \cite{large-sieve-note}] If $A \in \cA_{d,\kappa}$, then
\[
\Big(\prod_{p \in d} \frac{\kappa_p}{|p|-\kappa_p}\Big)\Tr(AJ) \le \Tr\Big(A\Big(\bigotimes_{p \in d}(|p|I_p-J_p)\bigotimes_{q\not\in d}J_q\Big)\Big).
\]
In other words, the matrix $B^d = \Big(\bigotimes_{p \in d}(|p|I_p-J_p)\bigotimes_{q\not\in d}J_q\Big) - \Big(\prod_{p \in d} \frac{\kappa_p}{|p|-\kappa_p}\Big)J$ is in $\cB_{d,\kappa}$.
\end{lem}
\begin{proof} Due to the multiplicative nature of the inequality, it's enough to check it in the case $d = \{p\}$ for just one partition $p$. So we just need to check that $|p|I_p - J_p - \frac{\kappa_p}{|p|-\kappa_p}J_p$ is in $\cB_{p,\kappa_p}$.

Multiplying through by $\frac{|p|-\kappa_p}{|p|}$, this is equivalent to $(|p|-\kappa_p)I_p - J_p \in \cB_{p,\kappa_p}$. If we delete any $\kappa_p$ rows and corresponding columns of this matrix, then due to the symmetry of $I_p,J_p$ the resulting matrix will be $(|p|-\kappa_p)I_{|p|-\kappa_p} - J_{|p|-\kappa_p}$, which is positive semidefinite by Cauchy-Schwarz.
\end{proof}

\begin{lem}[Montgomery \cite{large-sieve-note}] If $S(\alpha)$ is the symmetric matrix such that
\[
x^TS(\alpha)x = \Big|\sum_j e^{2\pi ij\alpha} x_j\Big|^2,
\]
that is, if $S(\alpha)_{jk} = \cos(2\pi\alpha(j-k))$, then
\[
\bigotimes_{p \in d}(|p|I_p-J_p)\bigotimes_{q\not\in d} J_q = \sum_{a \in (\ZZ/|d|)^\times} S\Big(\frac{a}{|d|}\Big).
\]
\end{lem}

\begin{thm}[Analytic Large Sieve Inequality \cite{montgomery-large}] If $S(\alpha) \in \RR^{Z\times Z}$ is defined as in the previous proposition, the $\alpha$s are $\delta$-spaced, and $Z$ is an interval with $|Z| = N$, then
\[
\Big\|\sum_\alpha S(\alpha)\Big\|_{op} \le N + \delta^{-1} - 1.
\]
\end{thm}

\begin{cor} The matrix $B^1 = (N+\delta^{-1}-1)I - \sum_{d\in D} \bigotimes_{p \in d}(|p|I_p-J_p)\bigotimes_{q\not\in d} J_q$ is positive semidefinite if the fractions with denominators $|d|$ for $d \in D$ are $\delta$-spaced. If the maximum $|d|$ for $d \in D$ is $Q$, then we can take $\delta^{-1} = Q(Q-1)$.
\end{cor}

Putting it all together, and scaling everything down by a factor of $\sum_{d\in D} \frac{\kappa(d)}{\phi_\kappa(d)}$, where $\phi_\kappa$ is the multiplicative function with $\phi_\kappa(p) = |p|-\kappa_p$, we get the following proof of the (sieve-theoretic) Large Sieve inequality:
\begin{align*}
\frac{1}{\sum_d \frac{\kappa(d)}{\phi_\kappa(d)}}\Bigg(\Big((N&+\delta^{-1}-1)I - \sum_{d\in D} \bigotimes_{p \in d}(|p|I_p-J_p)\bigotimes_{q\not\in d} J_q\Big)\\
&+ \sum_{d > 1}\Big(\Big(\bigotimes_{p \in d}(|p|I_p-J_p)\bigotimes_{q\not\in d}J_q\Big) - \frac{\kappa(d)}{\phi_\kappa(d)}J\Big)\Bigg) + J = \frac{N+\delta^{-1}-1}{\sum_d \frac{\kappa(d)}{\phi_\kappa(d)}}I.
\end{align*}

Note that from our computational point of view, the analytic large sieve inequality (which is the hardest part of the proof) is considered a triviality, since in any particular case we may easily use a computer to numerically verify that the matrix
\[
B^1 = (N+\delta^{-1}-1)I - \sum_\alpha S(\alpha) = (N+\delta^{-1}-1)I - \sum_{d\in D} \bigotimes_{p \in d}(|p|I_p-J_p)\bigotimes_{q\not\in d} J_q
\]
is positive semidefinite. In fact, the final set of matrices $B^d$ used makes no explicit mention of the matrices $S(\alpha)$ whatsoever - they were only introduced in order to facilitate the proof that $B^1$ is positive semidefinite.

\section{The Larger Sieve}\label{s-larger}

To treat Gallagher's Larger Sieve \cite{gallagher-larger} in this framework, it is necessary to introduce an additional matrix $B^0$ where the only constraint on the entries of $B^0$ is that they are nonnegative. This corresponds to the fact that for the matrices $A = xx^T$ of interest to us, the vector $x$ and the matrix $A$ will actually have positive entries (in fact, we are really only concerned with the case where the vector $x$ has all of its entries in $\{0,1\}$, but there doesn't seem to be any good way to make use of that fact). Here is Gallagher's sieve:
\[
B^0 + \sum_p \frac{\Big(\log(p)I_p - \frac{\log(p)}{p-\kappa_p}J_p\Big)\bigotimes_{q\ne p}J_q}{\sum_q \frac{\log(q)}{q-\kappa_q}-\log(N)} + J = \frac{\sum_p \log(p) - \log(N)}{\sum_p \frac{\log(p)}{p-\kappa_p}-\log(N)}I.
\]
The diagonal entries of $B^0$ are $0$, while the off-diagonal entries are given by
\[
B^0_{ij} = \frac{\log(N) - \sum_{p \mid i-j} \log(p)}{\sum_p \frac{\log(p)}{p-\kappa_p}-\log(N)}\text{ for }i\ne j.
\]
Note that the matrix $B^0$ is quite far from being positive semidefinite - every $2\times 2$ submatrix with corresponding rows and columns has the form $\begin{bmatrix} 0 & x\\ x & 0\end{bmatrix}$ for some $x \ge 0$.

In the case where $\kappa_p \approx \alpha p$ for some constant $0 \le \alpha < 1$, if we apply the Larger Sieve with the first $eN^{1-\alpha}$ primes, we get an upper bound of approximately $e(1-\alpha)N^{1-\alpha}\log(N^{1-\alpha})$ on the size of the sifted set $X$. In the special case $\kappa_p = 1$ of interest to us, we get the rather weak upper bound $eN\log(N)$ on the size of the subset $X$ of our interval $Z$ (which has size only $N$).

\section{A simpler semidefinite framework}\label{s-simpler}

Note that in both the Large Sieve and the Larger Sieve, the matrices $B^d$ were highly symmetric - in the sense that they could be represented as linear combinations of tensor products of $I$ and $J$ - for all $d$ other than $0,1$. Thus it seems likely that restricting our choices of $B^d$ to highly symmetric matrices for $d \ne 0,1$ does not lose too much power.

In this section we will show that testing whether $B^d \in \cB_d$ for such highly symmetric matrices essentially reduces to verifying a finite set of linear inequalities. In fact, we show that in most cases they come from the same types of simple applications of Cauchy-Schwarz that are used in the Large sieve and the Larger Sieve. This allows us to define a simpler semidefinite relaxation which still retains some of the power of the full semidefinite relaxation introduced in this paper, and which is significantly more practical for numerical computations.

\begin{defn} If $i,j \in Z$, then the expression $(j-i,\cP)$ is defined to be the set of partitions $p$ in $\cP$ such that $i$ and $j$ lie in the same part of $p$.
\end{defn}

\begin{defn} For $d \subseteq \cP$ we define $\phi(d)$ to be $\prod_{p \in d} (|p|-1)$, and we set $\mu(d) = (-1)^{\#\{p \in d\}}$. For $s \subseteq \cP$, we define $\phi^s_\kappa(d)$ by $\phi^s_\kappa(d) = \prod_{p \in d\cap s} (|p|-\kappa_p-1) \prod_{p \in d\setminus s} (|p|-1)$.
\end{defn}

\begin{thm}\label{sym-semidefinite} If a matrix $B$ is given by
\[
B_{ij} = b_{(j-i,\cP)}
\]
for a given system of real numbers $b_d$, then $B \succeq 0$ if the numbers $b_d$ satisfy the system of inequalities
\[
\forall k \subseteq \cP,\ \sum_{d} \phi(\cP\!\setminus\! (d\cup k))\mu(k\!\setminus\! d)\ b_d \ge 0.
\]
The above system of inequalities is satisfied iff there exist weights $w_k \ge 0$ such that
\[
b_d = \sum_k \phi(d\cap k)\mu(k\setminus d)w_k,
\]
or equivalently such that
\[
B = \sum_k w_k \bigotimes_{p \in k} \Big(|p|I_p - J_p\Big)\bigotimes_{q\not\in k}J_q.
\]
If the natural map $Z \rightarrow \prod_{p \in \cP} p$ is surjective, then $B \succeq 0$ if and only if the above conditions are satisfied.
\end{thm}
\begin{proof} We just need to prove this in the case $Z = \prod_{p \in \cP} p$. We identify each partition $p$ with a cyclic group $\ZZ/|p|$, and $Z$ with the product group $\prod_{p\in\cP}\ZZ/|p|$. We define the matrix $U \in \RR^{Z\times \hat{Z}}$ to have as columns the set of characters $\chi: Z \rightarrow \CC^\times$, so that $U_{i,\chi} = \chi(i)$. The matrix $B$ is positive semidefinite if and only if the matrix $U^HBU \in \RR^{\hat{Z}\times\hat{Z}}$ is positive semidefinite.

The $\psi,\chi$ entry of $U^HBU$ is given by
\begin{align*}
(U^HBU)_{\psi,\chi} &= \sum_{i,j} B_{ij}\overline{\psi(i)}\chi(j)\\
&= \sum_{i,j} b_{(j-i,\cP)}\overline{\psi(i)}\chi(j)\\
&= \sum_{x\in Z} b_{(x,\cP)} \sum_{i} \overline{\psi(i)}\chi(i+x)\\
&= \sum_{x \in Z} b_{(x,\cP)} \chi(x) |Z|\langle \psi, \chi\rangle.
\end{align*}
In particular, the matrix $U^HBU$ is diagonal, so in order to check that it is positive semidefinite we just need to check that the diagonal entries are nonnegative.

Every character $\chi$ on $Z$ can be written as a product of characters $\chi_p$ on $\ZZ/|p|$. Letting $k_\chi$ be the set of $p$ such that $\chi_p$ is nontrivial, we see that the $\chi,\chi$ entry of $U^HBU$ is proportional to
\begin{align*}
\sum_{x\in Z} b_{(x,\cP)} \chi(x) &= \sum_{d \subseteq \cP} b_d \sum_{(x,\cP) = d} \chi(x)\\
&= \sum_d b_d \prod_{p \in \cP\setminus d} \begin{cases} |p|-1 & p \not\in k_\chi\\ -1 & p \in k_\chi\end{cases}\\
&= \sum_d b_d\ \phi(\cP\!\setminus\!(d\cup k_\chi))\mu((\cP\!\setminus\! d)\cap k_\chi).
\end{align*}

Noting that the system of inequalities which we require the $b_d$s to satisfy has the same number of inequalities as variables, it is natural to look for a positive basis for the cone of solutions. By the multiplicative nature of these inequalities, it is enough to understand the case of just one partition $p$, in which case we see that a positive basis for the cone of solutions is given by $|p|I_p-J_p$ and $J_p$. Concretely, we can verify that the matrices
\[
\bigotimes_{p \in k} \Big(|p|I_p - J_p\Big)\bigotimes_{q\not\in k}J_q
\]
form a positive basis for the cone of solutions by verifying that for any $k,k' \subseteq \cP$, we have
\[
\sum_{d} \phi(\cP\!\setminus\! (d\cup k'))\mu(k'\!\setminus\! d)\cdot \phi(d\cap k)\mu(k\setminus d) = \begin{cases} 0 & k \ne k',\\ \prod_{p \in \cP}(|p|-2) & k = k'.\end{cases}\qedhere
\]
\end{proof}

Applying the above result to the case where some parts of some partitions have been knocked out, we have the following result.

\begin{cor} If $B^s$ is given by $B^s_{ij} = b^s_{(j-i,\cP)}$, then to check that $B^s \in \cB_{s,\kappa}$ it is sufficient to check that the system of linear inequalities
\[
\forall k \subseteq \cP,\ \sum_{d} \phi^s_\kappa(\cP\!\setminus\! (d\cup k))\mu(k\!\setminus\! d)\ b^s_d \ge 0
\]
is satisfied. This occurs iff there are weights $w^s_k \ge 0$ such that
\[
B^s = \sum_k w^s_k \bigotimes_{p \in k\cap s} \Big((|p|-\kappa_p)I_p - J_p\Big)\bigotimes_{q \in k\setminus s} \Big(|q|I_q - J_q\Big)\bigotimes_{r\not\in k}J_r,
\]
and in this case the $b^s_d$s are given by
\[
b^s_d = \sum_k \phi^s_\kappa(d\cap k)\mu(k\setminus d) w_k.
\]
\end{cor}

Examining the matrix
\[
\bigotimes_{p \in k\cap s} \Big((|p|-\kappa_p)I_p - J_p\Big)\bigotimes_{q \in k\setminus s} \Big(|q|I_q - J_q\Big)\bigotimes_{r\not\in k}J_r,
\]
we see that there is little point in introducing a positive weight $w^s_k$ for it if $k\setminus s$ is nonempty, since $|q|I_q - J_q$ is a positive combination of the matrices $(|q|-\kappa_q)I_q - J_q$ and $J_q$. This justifies considering the following simpler semidefinite relaxation.

\begin{defn} If $Z$ is a set, $\cP$ is a collection of partitions of $Z$, $\kappa_p$ are integers attached to the partitions $p \in \cP$, and $D_s, D_f$ are collections of subsets of $\cP$, then we define the $D_s,D_f$-\emph{relaxation} of the sifting problem to be the problem of finding the following minimum $\lambda$ such that there exist $B^d \in \cB_{d,\kappa}$ for $d \in D_s$ and $w_d \ge 0$ for $d \in D_f$ with
\[
\sum_{d \in D_s} B^d + \sum_{d \in D_f} w_d \bigotimes_{p \in d} \Big((|p|-\kappa_p)I_p - J_p\Big)\bigotimes_{q\not\in d}J_q + J = \lambda I.
\]
If $0 \in D_s$, then we also allow a matrix $B^0$ to be included, where the only constraint on $B^0$ is that its entries are nonnegative.
\end{defn}

The Large Sieve and the Larger Sieve can both be thought of as living in this framework, with $D_s = \{1\}$ for the Large Sieve and $D_s = \{0\}$ for the Larger Sieve. In the case of the Larger Sieve, $D_f$ contains only primes, and the correspoding weights $w_p$ are proportional to $\log(p)$. One weakness of taking $D_s = \{0\}$ as in the Larger Sieve is that if the set $Z$ contains any pair of elements $i \ne j$ such that $(i-j,\cP) = \cP$, then it will be impossible to find a system of weights $w_d \ge 0$ such that
\[
\sum_{d \in D_f} w_d \bigotimes_{p \in d} \Big((|p|-\kappa_p)I_p - J_p\Big)\bigotimes_{q\not\in d}J_q + J
\]
has nonpositive off-diagonal entries, since the $i,j$ entry will automatically be at least $1$.

In the case of the Large Sieve, if only prime $d$s are used in $D_f$ (as is common in many applications), then the weights $w_d$ are proportional to $\frac{p-\kappa_p}{p}$.

The special case $D_s = \{1\}$ is particularly interesting in the case $Z$ is an interval: if we have
\[
B^1 + \sum_{d \in D_f} w_d \bigotimes_{p \in d} \Big((|p|-\kappa_p)I_p - J_p\Big)\bigotimes_{q\not\in d}J_q + J = \lambda I,
\]
then $B^1$ will have constant diagonal entries, so it will be a symmetric Toeplitz matrix, and can potentially be checked for positive semidefiniteness with a specialized algorithm. If $\lambda$ is taken to be minimal, then $B^1$ will also be singular. In this case, we can apply the Pisarenko harmonic decomposition (see \cite{pisarenko-harmonic}, which uses Chapter 4 of \cite{grenander-toeplitz}) to write
\[
B^1 = V^HDV,
\]
where $D$ is a diagonal matrix with positive entries and dimension equal to the rank of $B^1$, and each row of $V$ has the form $(1, e^{2\pi i f}, ..., e^{2\pi i (N-1)f})$ for some frequency $f$. The set of frequencies is determined by $B^1$ via the fact that $B^1$ and $V$ must have the same nullspace. This decomposition of $B^1$, together with the weights $w_d$, can be used to give very short proofs of upper bounds on the sizes of admissible tuples contained in short intervals.

\section{The first level of the semidefinite hierarchy}\label{s-first-level}

The Large Sieve is often stated in a weaker form than the one given here, where we restrict the set of $d$s considered to the primes, and this weaker form gives comparable results. Additionally, the Larger Sieve makes no use of non-prime $d$ whatsoever, and in fact it is hard to modify it to take them into account in any useful way. Do we really need to use the matrices $B^d$ with $d$ non-prime to get good results? What happens if we restrict ourselves to just using the primes?

A preliminary result is that if $p$ corresponds to a partition into just two parts, then there is never any reason to consider $d = pk$ with $k$ nontrivial. As a consequence, if all partitions have two parts, then the semidefinite framework considered here gets the exact answer at just the first level of the semidefinite hierarchy (i.e., using only $B^p$s with $p$ corresponding to a single partition).

\begin{thm} Suppose that $p$ is a partition into two parts $c_0,c_1$, and that $D$ is a collection of subsets of $\cP\setminus \{p\}$ such that the expression \eqref{target}, when restricted to $c_i$, proves the bound $|X\cap c_i| \le \lambda_i$. Then using the collection of subsets $D \cup \{p\}$, the expression \eqref{target} proves the bound $|X| \le \max(\lambda_0,\lambda_1)$.
\end{thm}
\begin{proof} Suppose that the bound on $|X \cap c_i|$ is proved via a system of matrices $B^d_i$. Then we can use the following matrices to prove the desired bound on $|X|$:
\[
B^p = \begin{bmatrix} 0 & -J\\ -J & 0\end{bmatrix}, \;\;\; B^d = \begin{bmatrix} B^d_0 & 0\\ 0 & B^d_1\end{bmatrix},
\]
where the indices are assumed to be arranged so that $c_0$ comes before $c_1$, and the matrices $-J$ in $B^p$ are rectangular rather than square. Then we have
\[
B^p + \sum_{d \in D} B^d + J = \begin{bmatrix} \sum B^d_0 + J & 0\\ 0 & \sum B^d_1 + J\end{bmatrix},
\]
and the operator norm of this matrix is clearly equal to $\max(\lambda_0,\lambda_1)$.
\end{proof}

Now we consider a toy scenario where the partitions are ``orthogonal'', so $Z = \prod_{p \in \cP} p$. In this case it is easy to compute the true upper bound on $|X|$ (i.e., $|X| \le \phi(\cP)$) using a symmetry argument, but the linear sieve framework is unable to prove this bound without using sieve weights $\lambda_d$ supported on all subsets $d \subseteq \cP$. In the semidefinite framework, the situation is nicer: we only need to use matrices $B^p$ for $p$ corresponding to partitions, and in fact the space of systems of matrices $B^p$ that prove the true upper bound is enormous.

\begin{thm} If $Z = \prod_{p\in \cP} p$, then there is a high-dimensional set of systems $B^p \in \cB_p$ such that
\[
\sum_{p \in \cP} B^p + J = \phi(\cP) I.
\]
\end{thm}
\begin{proof} Due to the symmetry of the problem, we only need to consider matrices $B^p$ which have the property that $B^p_{ij}$ only depends on $(i-j,\cP)$. So suppose that each $B^p$ has the form
\[
B^p_{ij} = b^p_{(i-j,\cP)}.
\]


By Theorem \ref{sym-semidefinite}, if we define $\phi^p(d)$ by
\[
\phi^p(d) = \begin{cases} \phi(d) & p \not\in d,\\ \frac{|p|-2}{|p|-1}\phi(d) & p \in d,\end{cases}
\]
then in order for the system of matrices $B^p$ to prove the bound $|X| \le \phi(\cP)$, we need the system of numbers $b^p_d$ to satisfy the following system of equations and inequalities:
\begin{align}
\sum_{d} \phi^p(\cP\!\setminus\! (d\cup k))\mu(k\!\setminus\! d)\ b^p_d &\ge 0\text{ for all }p, k,\label{ell}\\
\sum_p b^p_d + 1 &= 0\text{ for all }d \ne \cP,\label{sums-2}\\
\sum_p b^p_{\cP} + 1 &= \phi(\cP).\label{sums-3}
\end{align}
Summing the first inequality above over $k' \subseteq k$ with weight $\phi^p(k')$ for some fixed $k$, we get
\[
\sum_{d \supseteq k} \phi^p(\cP\!\setminus\! d)\ b^p_d \ge 0\text{ for all }p, k.
\]

Summing the above inequality over $p \not\in k$ with weights $\phi^\cP(\cP\setminus pk)\phi(\cP\setminus p)^{-1}$ where $\phi^\cP(k) = \prod_{q \in k} (|q|-2)$, we get
\begin{align*}
0 \le &\sum_{p \not\in k} \phi^\cP(\cP\!\setminus\!pk)\phi(\cP\!\setminus\!p)^{-1} \sum_{d \supseteq k} \phi^p(\cP\!\setminus\! d)\ b^p_d\\
&= \sum_{p \not\in d} (b^p_{pd} + (|p|-2)b^p_{d})\phi(d)^{-1} \sum_{k \subseteq d} \phi^\cP(\cP\!\setminus\!pk)\\
&= \sum_{d} \phi^\cP(\cP\!\setminus\!d) \sum_p b^p_d\\
&= \Big(\sum_p b^p_\cP + 1\Big) - \sum_d \phi^\cP(\cP\!\setminus\!d)\\
&= \Big(\sum_p b^p_\cP + 1\Big) - \phi(\cP) = 0,
\end{align*}
so all the inequalities with $p \not\in k$ must be satisfied with equality. This implies the system of equations
\[
b^p_{pd} + (|p|-2)b^p_d = 0\text{ for all }p \not\in d.
\]
If we impose these extra equations, then the inequalities $\sum_{d \supseteq k} \phi^p(\cP\!\setminus\! d)\ b^p_d \ge 0$ with $p \not\in k$ will be satisfied with equality, and the equations $\sum_p b^p_d + 1 = 0$ for $d \ne \cP$, $\sum_p b^p_{\cP} + 1 = \phi(\cP)$ will become linearly dependent modulo these extra equations.

For the remaining inequalities, we introduce variables $w^p_k$ such that
\[
b^p_d = \sum_k \phi^p(d\cap k)\mu(k\setminus d) w^p_k,
\]
and the remaining inequalities \eqref{ell} will be satisfied iff we have $w^p_k \ge 0$ for all $p,k$. The system of identities $b^p_{pd} + (|p|-2)b^p_d = 0$ for all $p \not\in d$ is equivalent to the system of identities $w^p_k = 0$ for all $p \not\in k$, since both are equivalent to \eqref{ell} being satisfied with equality for all $p \not\in k$.

Summing \eqref{sums-2} and \eqref{sums-3} with weights $\phi(\cP\!\setminus\! (d\cup j))\mu(j\!\setminus\! d)$ and replacing $b^p_d$s by their definitions in terms of the $w^p_k$s, we get
\begin{align*}
\phi(\cP) - \sum_d \phi(\cP\!\setminus\! (d\cup j))\mu(j\!\setminus\! d) &= \sum_{d,p} \phi(\cP\!\setminus\! (d\cup j))\mu(j\!\setminus\! d) b^p_d\\
&= \sum_{k,d,p} \phi(\cP\!\setminus\! (d\cup j))\mu(j\!\setminus\! d) \phi^p(d\cap k)\mu(k\setminus d) w^p_k\\
&= \sum_{p, k} w^p_k \sum_d \phi(\cP\!\setminus\! (d\cup j))\mu(j\!\setminus\! d) \phi^p(d\cap k)\mu(k\setminus d)\\
&= \sum_{p \in j} w^p_j \frac{|p|-1}{|p|} |\cP| - \sum_{p \not\in j} w^p_{pj} \frac{1}{|p|} |\cP|.
\end{align*}
Thus, for $j \ne \emptyset$ we have
\[
\sum_{p \in j} w^p_j \frac{|p|-1}{|p|} = \prod_p \frac{|p|-1}{|p|} + \sum_{p \not\in j} w^p_{pj} \frac{1}{|p|},
\]
and for $j = \emptyset$ we have
\[
1 = \prod_p \frac{|p|-1}{|p|} + \sum_{p} w^p_{p} \frac{1}{|p|},
\]
and this last equation follows from the equations for $j \ne \emptyset$ due to the linear dependence found earlier. We can now pick positive values for the $w^p_j$s by a downward induction on $j$, starting with $j = \cP$, and at each step of the inductive construction we find that we have a high dimensional set of choices.
\end{proof}

\bibliographystyle{plain}
\bibliography{sieve}

\end{document}